\documentclass[a4paper]{article}

\usepackage[utf8]{inputenc}
\usepackage{float, amsmath, amsthm, amssymb, tikz, url}

\usepackage{marginnote}

\newcommand{\Z}[1]{\mathbb{Z}^{#1}}
\newcommand{\Fg}[1]{\mathbb{F}_{#1}}
\newcommand{\N}{\mathbb{N}}
\newcommand{\A}{\mathcal A}

\DeclareMathOperator{\supp}{supp}
\DeclareMathOperator{\orb}{orb}
\DeclareMathOperator{\stab}{stab}

\theoremstyle{plain}
\newtheorem{theorem}{Theorem}[section]
\newtheorem{corollary}[theorem]{Corollary}

\newtheorem{proposition}[theorem]{Proposition}
\newtheorem{conjecture}[theorem]{Conjecture}

\theoremstyle{definition}
\newtheorem{definition}[theorem]{Definition}

\theoremstyle{remark}
\newtheorem{remark} [theorem]{Remark}

\newtheorem{example}[theorem]{Example}

\title{Necessary conditions for tiling finitely generated amenable groups}

\author{Benjamin Hellouin de Menibus\footnote{This article was written during stays funded by an LRI internal project.}\\ Laboratoire de Recherche en Informatique \\ Université Paris-Sud - CNRS - CentraleSupélec,\\ Université Paris-Saclay, France\\ \raisebox{-2pt}{\includegraphics[height=12pt]{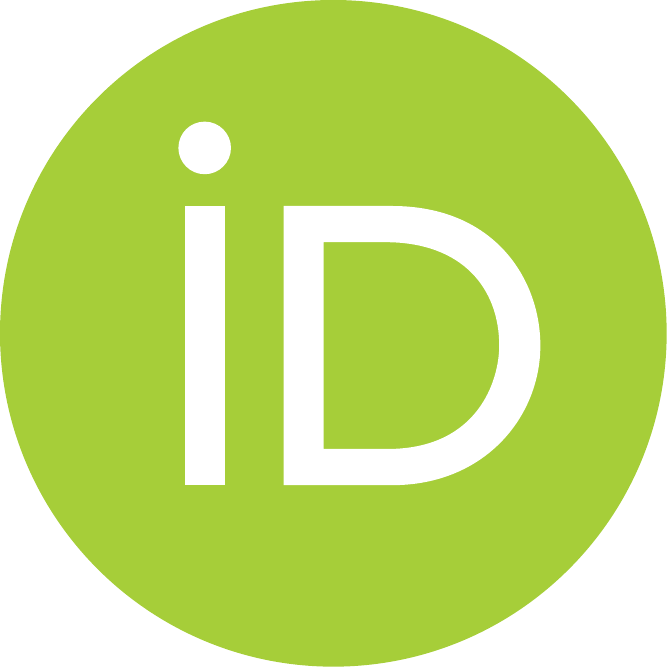}} \url{https://orcid.org/0000-0001-5194-929X}\\\url{hellouin@lri.fr}\\~\\
Hugo Maturana Cornejo\footnote{The second author would like to thank Nathalie Aubrun for the support offered during his stay at ENS Lyon. This article was partially funded by the ECOS-SUD project C17E08,
the ANR project CoCoGro (ANR-16-CE40-0005) and CONICYT doctoral fellowship 21170770.}\\
Departamento de ingeniería matemática, DIM-CMM\\Universidad de Chile\\\url{hmaturana@dim.chile.cmm}}

\begin{document}
\maketitle
\begin{abstract}
We consider a set of necessary conditions which are efficient heuristics for deciding when a set of Wang tiles cannot tile a group. 

Piantadosi \cite{SP08} gave a necessary and sufficient condition for the existence of a valid tiling of any free group. This condition is actually necessary for the existence of a valid tiling for an arbitrary finitely generated group.

We then consider two other conditions: the first, also given by Piantadosi \cite{SP08}, is a necessary and sufficient condition to decide if a set of Wang tiles gives a strongly periodic tiling of the free group; the second, given by Chazottes et. al. \cite{CG14}, is a necessary condition to decide if a set of Wang tiles gives a tiling of $\mathbb Z^2$.

We show that these last two conditions are equivalent. Joining and generalising approaches from both sides, we prove that they are necessary for having a valid tiling of any finitely generated amenable group, confirming a remark of Jeandel \cite{EJ15b}.\medskip

MCS classification : Primary 37B50; Secondary 37B10, 05B45.\medskip

\end{abstract}

\section{Introduction}

$\Z{2}$-subshifts of finite type (SFT) are a set of colorings of the $2$-dimensional lattice $\Z2$, or \emph{tilings}, defined by a finite set of local restrictions. There are various equivalent ways to express the restrictions, such as the Wang tiles formalism introduced by Hao Wang \cite{WH61}. This formalism was introduced to study the \emph{domino problem}: given as input a set of restrictions (e.g. a set of Wang tiles), is there an algorithm that decides whether there is a tiling of $\Z2$ that respects those restrictions? 

R. Berger \cite{BER66} showed that the domino problem is undecidable. The proof depends heavily on notions of periodicity and aperiodicity, more precisely on the existence of a set of Wang tiles that only tile $\Z2$ in a strongly aperiodic manner. This is in stark contrast with the situation on $\Z{}$ where the domino problem is decidable thanks to a graph representation \cite{LM95}.

There has been a recent interest in symbolic dynamics on more general contexts, such as where the lattice $\Z2$ is replaced by the Cayley graph of an infinite, finitely generated group. Using again the existence of strongly aperiodic SFTs, the domino problem was shown to be undecidable, apart from $\Z{d}$, on some semisimple Lie groups \cite{SM97}, the Baumslag-Solitar groups \cite{AK13}, the discrete Heisenberg group (announced, \cite{MSU14}), surface groups \cite{CG17, AB18}, semidirect products on $\Z2$ \cite{BS18} or some direct products \cite{SB17}, polycyclic groups \cite{EJ15}, some hyperbolic groups \cite{CG17b}\dots  It is decidable on free groups \cite{SP08} and on virtually free groups \cite{BS13}, and it is conjectured that these are the only groups where the domino problem is decidable (Conjecture~\ref{conj:1} below).

As a consequence, outside of free and virtually free groups, one can not expect to find simple necessary and sufficient conditions for admitting a valid tiling. However, heuristics can be very useful when making an exhaustive search for SFTs with desired properties; necessary conditions in particular allow fast rejection of most empty SFTs. For example, a transducer-based heuristic was used in the search for the smallest set of Wang tiles that yield a strongly aperiodic $\Z2$-SFT \cite{JR15}. It is also of theoretical interest to understand how the group properties impact necessary conditions.

\subsection{Statements of results}
\label{ss:main statements}

We first consider a necessary and sufficient condition introduced by Piantadosi for an SFT on the free group to admit a valid tiling \cite{SP08}. It is well-known that an SFT on a finitely generated group can only admit a tiling if the ``corresponding'' SFT on the free group does, so this becomes a necessary condition on an arbitrary f.g. group (Corollary~\ref{cor:arbitrary}).

The next two stronger conditions were introduced by Piantadosi (to decide if an SFT admits a strongly periodic tiling of the free group) and by Chazottes-Gambaudo-Gautero \cite{CG14} in a more general context of tiling the plane by polygons, but which is necessary for an SFT to admit a tiling of $\mathbb Z^2$ \cite{JV19}. We prove that the two conditions are equivalent (Theorem~\ref{thm:equivalence}), and that they form a necessary condition for an SFT to admit a valid tiling on any amenable group (Theorem~\ref{thm:amenable}), confirming a remark of Jeandel (\cite{EJ15b}, Section~3.1).

Finally, we provide for any non-free finitely generated group a counterexample that satisfies all conditions but does not provide a valid tiling.
\section{Preliminaries}
\label{s:preliminaries}

\subsection{Symbolic dynamics on groups}

In the whole article $G$ is an infinite, finitely generated group with unit element $1_{G}$. We write $G=\langle \mathcal S\mid \mathcal R\rangle$ where $\mathcal S = \{g_{1},\dots, g_{d}\}$ is a finite set of generators and $\mathcal R = \{r_{1},\dots, r_{m},\dots\}\subset (S\cup S^{-1})^\ast$ is a (possibly infinite) set of relations. By convention $r\in\mathcal R$ means that $r = 1_G$.

For instance:
\begin{itemize}
\item the free group $\Fg{d}$ is the group on $d$ generators with no relations;
\item $\mathbb{Z}^{2}=\langle \{g_1, g_2\}\mid g_1g_2g_1^{-1}g_2^{-1}\rangle$.
\end{itemize}

Let $\mathcal{A}$ be a finite set endowed with the discrete topology; denote its cardinality $\#\mathcal A$. Let $\mathcal{A}^{G}=\{(x_{g})_{g\in G}\,|\,\forall g\in G :x_{g}\in\mathcal{A}\}$ be the set of all functions from $G$ to $\mathcal{A}$ endowed with the product topology. Given a finite subset $F\subset G$, an element $P\in\mathcal{A}^{F}$ is called a \emph{pattern} and $F=\supp(P)$ its \emph{support}; the set of all patterns is denoted $\mathcal A^\ast$. 

$\mathcal{A}^{G}$ is a compact space called the $G$-full shift. It is a symbolic dynamical system under the following $G$-action, called the \emph{$G$-shift}: 
\[
\forall x\in\mathcal{A}^{G}, \forall h\in G, (\sigma_{h}(x_{g}))_{g\in G}=(x_{h^{-1}g})_{g\in G}
\]
We call $G$-subshift a closed shift-invariant subset $Y\subset\mathcal{A}^{G}$. 

A pattern $P\in\mathcal{A}^F$ is said to \emph{appear} in a configuration $x\in\mathcal{A}^{G}$  (and we write $P\sqsubset x$) if there exists $g\in G$ such that $\sigma_g(x)|_{F}=P$.

Given a set of forbidden patterns $\mathcal F\subset \mathcal A^\ast$, we can define the corresponding $G$-subshift:
\[Y=Y_{\mathcal{F}}=\{x\in\mathcal{A}^{G}\mid \forall P\sqsubset x : P\notin \mathcal{F}\}.\]

Every $G$-subshift can be defined in this way using a set of forbidden patterns. When a subshift can be defined by a \emph{finite} set of forbidden patterns, we say it is a $G$-subshift of finite type ($G$-SFT). If furthermore the set of forbidden patterns can be chosen so that every pattern in $\mathcal{F}$ has support of the form $\{1_{G},g_{i}\}$ where $g_{i}\in\mathcal S$ for some set of generators $\mathcal S$, we say it is a $G$-nearest-neighbor subshift of finite type ($G$-NNSFT). Notice that this definition depends on the choice of $\mathcal S$ which is usually clear in the context. 

For example, If we consider $G=\mathbb{Z}$ with generator $+1$, $\mathcal{A}=\{0,1\}$ and $\mathcal{F}=\{11\}$ we obtain a $\mathbb{Z}$-NNSFT, the golden mean shift, a classical example in symbolic dynamics. 

\begin{definition}[Weakly \& strongly aperiodic]
For a configuration $x\in\mathcal{A}^{G}$, we define the orbit of the element $x$ under the shift action as $\orb_{\sigma}(x)=\{\sigma_{g}(x)|g\in G\}$ and the set of elements on $G$ that fix the configuration $x$ by $\stab_{\sigma}(x)=\{g\in G| \sigma_{g}(x)=x\}$. 
A configuration $x\in\mathcal{A}^{G}$ is 
\begin{description}
\item[strongly periodic] if $\stab_\sigma(x)$ has finite index or, equivalently, if $\orb_{\sigma}(x)$ is finite;
\item[strongly aperiodic] if $\stab_{\sigma}(x)=\{1_{G}\}$.
\item[weakly periodic] if it is not strongly aperiodic;
\item[weakly aperiodic] if it is not strongly periodic.
\end{description}
More generally, a subshift $X\subset\mathcal{A}^{G}$ is weakly/strongly aperiodic if every configuration on $X$ is weakly/strongly aperiodic.
\end{definition}

\begin{example}\label{ex:NNSFT}
In $G = \Z2$,
\begin{itemize}
\item the configuration $x$ such that $x_g=0$ for all $g$ is strongly periodic;
\item the configuration $x$ such that $x_{g_{1}^{n}}=0$ for all $n$, and $x_g=1$ otherwise, is weakly periodic and weakly aperiodic;
\item the configuration $x$ such that $x_{(0,0)}=0$, and $x_g=1$ otherwise, is strongly aperiodic.
\end{itemize}
\end{example}

\subsection{Wang tiles, NNSFT and graphs}

\begin{definition}[Wang tiles, Wang subshifts]
Let $G = \langle S\mid \mathcal R \rangle$ be a finitely generated group and $\mathcal{C}$ a finite set of colors. A \emph{Wang tile} on $\mathcal C$ and $S$ is a map $S\cup S^{-1} \rightarrow \mathcal{C}$.\medskip

Given a set $T$ of Wang tiles, the corresponding $G$-Wang subshift is defined as:
\[
X_{T}=\{(x_{g}) \in T^G\, |\,  \forall g\in G, s\in S\cup S^{-1}, x_{g}(s)=x_{gs}(s^{-1})\}
\]
We call the elements in $X_{T}$ $G$-Wang tilings.
\end{definition}

Notice that the definition of a Wang tile depends only on the chosen set of generators, so that the same Wang tile can be used for $\Fg{2}$ and $\mathbb Z^2$, for example.

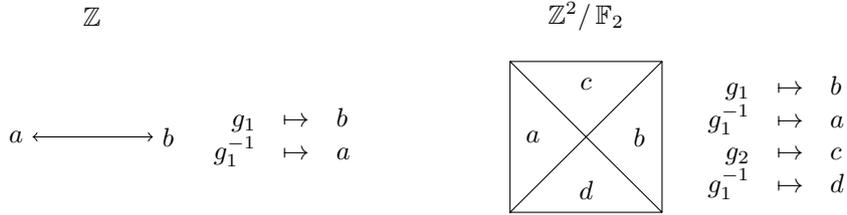
\begin{figure}[h!]
\begin{center}
\begin{tikzpicture}[every text node part/.style={align=right}]
\node at (1,2.6) {$\mathbb Z$};
\node (a) at (0,1) {$a$};
\node (b) at (2,1) {$b$};
\draw[<->] (a) -- (b);
\node at (3.5,1) {$\begin{array}{rcl}g_1&\mapsto& b\\g_1^{-1}&\mapsto& a\end{array}$};

\begin{scope}[shift={(6.5,0)}]
\node at (1,2.6) {$\mathbb Z^2 /\, \Fg{2}$};
\draw (0,0) -- (2,0) -- (2,2) -- (0,2) -- (0,0) -- (2,2) (0,2) -- (2,0);
\node at (0.3,1) {$a$};
\node at (1.7,1) {$b$};
\node at (1,1.7) {$c$};
\node at (1,0.3) {$d$};
\node at (3.5,1) {$\begin{array}{rcl}g_1&\mapsto& b\\g_1^{-1}&\mapsto& a\\g_2&\mapsto& c\\g_1^{-1}&\mapsto& d\end{array}$};
\end{scope}
\end{tikzpicture}

     \label{fig:wang}
     \caption{Examples of Wang tiles with colors $\mathcal{C}=\{a,b,c,d\}$ on one and two generators, respectively, with their corresponding maps.}
\end{center}
\end{figure}

Take any $G$-NNSFT $X$ on alphabet $\mathcal A$, where $G = \langle\{g_1,\dots, g_d\}\mid\mathcal R\rangle$ is an arbitrary finitely generated group. Let $\mathcal F$ be a set of forbidden patterns with each support of the form $\{1_{G},g_{i}\}$.

We associate to $X$ a set of $d$ graphs $\Gamma_1,\dots, \Gamma_d$, where the set of vertices is $\mathcal A$ for all $\Gamma_i$, and 

\[\forall a,b\in \mathcal A,\qquad a\to b \text{ in }\Gamma_i \Longleftrightarrow \left\{\begin{array}{ccc}1_G&\to& a\\g_i&\to& b\end{array}\right.\notin\mathcal F.\]

By definition of a $G$-NNSFT, it follows that a configuration $x$ belongs to $X$ if, and only if, $x_h \to x_{hg_i}$ is an edge in $\Gamma_i$ for all $h\in G$ and all $1\leq i\leq d$.

\begin{definition}[Cycles]
A \emph{cycle} on a graph $\Gamma$ is a path - with possible edge and vertex repetitions - that starts and ends on the same vertex. A cycle through the vertices $a_{1},\ldots,a_{n}a_{1}$, with $a_{i}\in\mathcal{A}$, is denoted $\overline{a_{1},\ldots,a_{n}}$. 

A cycle is \emph{simple} if it does not contain any vertex repetition. Denote $\mathcal{SC}(\Gamma)$ the set of simple cycles on $\Gamma$, which is a finite set.
\end{definition}

\begin{remark}
In graph theory, cycles are sometimes called \emph{closed walks}, in which case cycle means simple cycle. We decided to follow Piantadosi's conventions \cite{SP08} for convenience.
\end{remark}

Let $w$ be a cycle and $a\in \mathcal{A}$. We define:

\[|w|_a=\#\{i\mid w_{i}=a, 1\leq i\leq |w| \}.\]

In any cycle, the path between the closest repetitions is a simple cycle. By removing this simple cycle and iterating the argument, we can see that any cycle $w$ can be decomposed into simple cycles, in the sense that there are integers $\lambda_\omega$ for $\omega\in \mathcal{SC}(\Gamma)$ such that:

\[\forall a\in \mathcal A, |w|_a = \sum_{\omega\in \mathcal{SC}(\Gamma)}\lambda_\omega |\omega|_a.\]

We say that two $G$-subshifts $X, Y\subset \mathcal A^G$ are (topologically) \emph{conjugate} if there is a shift-commuting homeomorphism $\Phi$ (that is, $\Phi\circ \sigma_g = \sigma_g\circ\Phi$ for all $g\in G$) such that $\Phi(X)=Y$. An shift-commuting homeomorphism (or \emph{conjugacy}) corresponds to a reversible cellular automaton: there is a finite subset $H\subset G$ and a local rule $\varphi : \A^H\to \A$ such that 
\[\forall x\in X, \forall g\in G,\ \Phi(x)_g = \varphi(\sigma_{g^{-1}}(x)|_H),\]
and $\Phi^{-1}$ is itself a cellular automaton.

\begin{proposition}\label{prop:conjugate}
For any set of generators, each $G$-SFT is conjugate to a $G$-NNSFT and each $G$-NNSFT is conjugate to a $G$-Wang subshift.
\end{proposition}

This is folklore. A detailed proof for the SFT - NNSFT part can be found in \cite{BL17} (Propositions~1.6 and 1.7), and a proof of the NNSFT - Wang subshift part in \cite{MH18}.

Since the conjugacy from a $G$-Wang subshift to a $G$-NNSFT can be chosen letter-to-letter (i.e. $H=\{1_G\}$ in the definition), it is easy to see that the conjugacy does not depend on $G$, so we could say that a set of graphs and a set of Wang tiles are conjugate.

\begin{proposition}
Let $X$ and $Y$ be two conjugate $G$-subshifts. $X$ admits a valid tiling if and only if $Y$ admits a valid tiling. The same is true for weakly/strongly (a)periodic tilings.
\end{proposition}

\section{Piantadosi's and Chazottes-Gambaudo-Gautero's conditions}

\subsection{State of the art on the free group and $\Z{2}$}

The first two condition were introduced by Piantadosi in the context of symbolic dynamics on the free group $\Fg{d}$.

\begin{definition}[\cite{SP08}]
A family of graphs $\Gamma=\{\Gamma_{i}\}_{1\leq i\leq d}$ on alphabet $\A$ satisfies condition $(\star)$ if and only if there is some nonempty $\A^{\prime}\subset \A$ with a \emph{coloring function} $\Psi : \mathcal \A^\prime\times\mathcal S\to\mathcal \A^\prime$ such that, for any color $a\in \mathcal \A^{\prime}$ and any generator $g_i\in S$,
$a \to \Psi(a,g_i) \text{ is an edge in } \Gamma_i$.
\end{definition}

\begin{theorem}[\cite{SP08}]\label{thm:star}
Let $X$ be a $\mathbb{F}_{d}$-NNSFT on alphabet $\A$. $X$ is nonempty if and only if the corresponding set of graphs satisfies condition $(\star)$.
\end{theorem}

This theorem provides a decision procedure for Domino problem in free groups of any rank: find a subalphabet such that every letter admits a valid neighbourg in the subalphabet for every generator.

\begin{definition}[\cite{SP08}]\label{def:piantadosi}
\label{def:star}
Consider a family of graphs $\Gamma=\{\Gamma_{i}\}_{1\leq i\leq d}$ and $\mathcal{SC}(\Gamma_i) = \{\omega^j_{i}\}_{1\leq j\leq \#\mathcal{SC}(\Gamma_i)}$ the set of simple cycles for each graph $\Gamma_i$.

We denote by $(\star\star)$ the following equation on $x_{i,j}$:
\[\forall a\in\mathcal{A},\ \sum_{j=1}^{ \#\mathcal{SC}(\Gamma_1)} x_{1,j} |\omega_{1}^{j}|_a = \sum_{j=1}^{ \#\mathcal{SC}(\Gamma_2)} x_{2,j} |\omega_{2}^{j}|_a = \dots = \sum_{j=1}^{ \#\mathcal{SC}(\Gamma_d)} x_{d,j} |\omega_{d}^{j}|_a.\]

We say that the graph family satisfies condition $(\star\star)$ if equation $(\star\star)$ is not empty (e.g. all graphs contain at least a cycle) and admits an nontrivial positive solution.
\end{definition}

\begin{remark}
We formulated the previous condition in terms of simple cycles (using the formalism from Theorem~3.6 instead of Theorem~3.4 in \cite{SP08}) because they form a finite set, making it easier to prove formally when the condition is not satisfied.
\end{remark}

\begin{theorem}[\cite{SP08}, Theorem~3.6]\label{thm:piantadosi} A $\Fg{d}$-NNSFT contains a strongly periodic configuration if and only the associated family of graphs satisfies condition $(\star\star)$.
\end{theorem}

\begin{example}\label{ex:piantadosi}We illustrate Piantadosi's conditions on the following example:

\begin{center}
		\begin{tikzpicture}
			\foreach \i in {0,1,2} {
				\node (V\i) at (150-120*\i:2cm) [shape=circle,draw=black] {$\i$};
			}
			\draw[-latex] (V0) [bend left] to (V1);
			\draw[-latex] (V1) [bend left] to (V2);
			\draw[-latex] (V2) [bend left] to (V0);
			
			\foreach \i in {0,1,2} {
				\node (W\i) at ([xshift=7cm,yshift=-1cm]90+120*\i:2cm) [shape=circle,draw=black] {$\i$};
			}
			\draw[-latex] (W1) to (W0);
			\draw[-latex] (W0) to (W2);
			\draw[-latex] (W1) [out=240,in=300,looseness=8] to (W1);
			\draw[-latex] (W2) [out=240,in=300,looseness=8] to (W2);
			
			\node at (150:2cm) [xshift=-1cm] {$\Gamma_1:$};
			\node at (150:2cm) [xshift=6cm] {$\Gamma_2:$};
		\end{tikzpicture}
\end{center}

The corresponding $\Fg{2}$-NNSFT admits a tiling, because it satisfies condition $(\star)$ on alphabet $\A^\prime = \A$. However, it does not admit a periodic tiling: the simple cycles of $\Gamma_1$ are (up to shifting) $\{\overline{012}\}$ and the simple cycles of $\Gamma_2$ are $\{\overline{1}, \overline{2}\}$, so Equation $(\star\star)$ is:
\begin{align*}x_{1,1} &= 0 &(a = 0)\\ x_{1,1} &= x_{2,1} &(a = 1)\\ x_{1,1} &= x_{2,2}&(a = 2)\end{align*}
which obviously doesn't admit a solution. As we will see later, the corresponding $\Z2$-NNSFT doesn't admit any tiling.
\end{example}

\begin{remark} \label{rem:common}
For example, if all graphs $\Gamma_i$ share a common cycle $w$, then condition $(\star\star)$ admits a solution and therefore the corresponding $\Fg{d}$-NNSFT contains a periodic configuration.
\end{remark}

\begin{definition}[\cite{CG14}]
\label{def:starstar}
Let $T$ be a set of Wang tiles on colors $\mathcal C$ and set of generators $S$. 
For each $g\in S\cup S^{-1}$ and each color $c\in\mathcal C$, define $c_g$ the subset of Wang tiles $\tau_i\in T$ such that $\tau_i(g)=c$.
We call $(\star\star)^\prime$ the following equation :
\[\forall g\in S, \forall c\in\mathcal C, \sum_{\tau_i\in c_g} x_i= \sum_{\tau_j\in c_{g^{-1}}}x_j\]
We say that $T$ satisfies condition $(\star\star)^\prime$ if Equation~$(\star\star)^\prime$ admits a positive nontrivial solution.
\end{definition}

\begin{theorem}[\cite{CG14}]\label{thm:starstarprime}
If a set $T$ of Wang tiles admits a valid tiling of $\mathbb Z^2$, then it satisfies condition $(\star\star)^\prime$.
\end{theorem}

This condition and result was introduced in \cite{CG14}, but a much easier presentation in our context is given in \cite{JV19}.

\begin{example}
Example~\ref{ex:piantadosi} is conjugate to the following set of Wang tiles:
\begin{center}
\begin{tikzpicture}[every text node part/.style={align=left}]
\node at (9,1) {$0\mapsto \tau_0$ \\ $1\mapsto \tau_1$ \\ $2\mapsto \tau_2$};
\begin{scope}[shift={(0,0)}]
\node at (1,2.6) {$\tau_0$};
\draw (0,0) -- (2,0) -- (2,2) -- (0,2) -- (0,0) -- (2,2) (0,2) -- (2,0);
\node at (0.3,1) {$a$};
\node at (1,1.7) {$b$};
\node at (1.7,1) {$b$};
\node at (1,0.3) {$a$};
\end{scope}
\begin{scope}[shift={(2.7,0)}]
\node at (1,2.6) {$\tau_1$};
\draw (0,0) -- (2,0) -- (2,2) -- (0,2) -- (0,0) -- (2,2) (0,2) -- (2,0);
\node at (0.3,1) {$b$};
\node at (1,1.7) {$a$};
\node at (1.7,1) {$c$};
\node at (1,0.3) {$a$};
\end{scope}

\begin{scope}[shift={(5.4,0)}]
\node at (1,2.6) {$\tau_2$};
\draw (0,0) -- (2,0) -- (2,2) -- (0,2) -- (0,0) -- (2,2) (0,2) -- (2,0);
\node at (0.3,1) {$c$};
\node at (1,1.7) {$b$};
\node at (1.7,1) {$a$};
\node at (1,0.3) {$b$};
\end{scope}

\end{tikzpicture}
\end{center}
Equation $(\star\star)^\prime$ becomes the following, where next to each equation is the corresponding generator and color:
\begin{align*}
(g_1, a)\qquad& x_{2}=x_{0}&(g_2, a)\qquad& x_{1}=x_{0}+x_1\\
(g_1, b)\qquad& x_{0}=x_{1}&(g_2, b)\qquad& x_{0}+x_2=x_{2}\\
(g_1, c)\qquad& x_{1}=x_{2}&(g_2, c)\qquad& 0=0\\
\end{align*}
This equation does not admit a positive nontrivial solution, so the corresponding $\Z2$-Wang subshift is empty.
\end{example}

\subsection{Conditions $(\star\star)$ and $(\star\star)'$ are equivalent}

Although conditions $(\star\star)$ and $(\star\star)'$ were introduced in very different contexts (periodic tilings of the free group and tilings of the Euclidean plane, respectively), it turns out that they are equivalent. The fact that $(\star\star)$ is a condition on graphs (NNSFTs) and $(\star\star)^\prime$ is a condition on sets of Wang tiles (Wang subshifts) is only cosmetic since Proposition~\ref{prop:conjugate} lets us go from one model to the other.

\begin{theorem}\label{thm:equivalence}
Let T be a set of Wang tiles over the set of colors $\mathcal C$ and the set of generators $S$.

$T$ satisfies condition $(\star\star)^\prime$ if, and only if, the associated graphs satisfy condition $(\star\star)$.
\end{theorem}

\begin{proof}
$(\Leftarrow)$ Let $(x_{i,j})$ be a nonnegative solution to equation $(\star\star)$. For every tile $\tau_i$, put $x_i = \sum_{j=1}^{\#\mathcal{SC}(\Gamma_1)}x_{1,j}|\omega_1^j|_{\tau_i}$.

Because each simple cycle of $\Gamma_1$ is a cycle, it contains as many tiles in $c_{g_1}$ as in $c_{g_1^{-1}}$; that is, $\sum_{\tau_i\in c_{g_1}} |\omega_1^j|_{\tau_i}  = \sum_{\tau_j\in c_{g_1^{-1}}} |\omega_1^j|_{\tau_i}$. Summing over all simple cycles $\omega_1^j$, we get $\sum_{\tau_i\in c_{g_1}} x_i  = \sum_{\tau_i\in c_{g_1^{-1}}} x_j$. 

Since $(x_{i,j})$ is a solution to Equation~$(\star\star)$, we also have $x_i = \sum_{j=1}^{\#\mathcal {SC}(\Gamma_n)}x_{n,j}|\omega_n^j|_{\tau_i}$ for every $n$, so the same argument shows that $(x_i)$ is a nonnegative solution of equation $(\star\star)^\prime$.\bigskip

$(\Rightarrow)$ Because equation $(\star\star)^\prime$ admits a solution, it admits a rational solution, and therefore an integer solution. Let $(x_i)$ be an integer, nonnegative solution of equation $(\star\star)^\prime$.

For the generator $g_1$, consider the graph $\Gamma_1$ obtained by the letter-to-letter conjugacy of Proposition~\ref{prop:conjugate}: concretely, it is the graph on vertices $\{\tau_i\}_{1\leq i\leq n}$ with $\tau_i\to\tau_j \Leftrightarrow \exists c\in\mathcal C, \tau_i \in c_{g_1} \text{ and }\tau_j \in c_{g_1^{-1}}$.
 
 We define an auxiliary graph $G_1$ on vertices $\{\tau_i^k\}_{1\leq i\leq n, 1\leq k \leq x_i}$ (that is, $x_i$ copies for each tile $\tau_i$) as follows.  

Because \[\forall c\in\mathcal C, \sum_{\tau_i\in c_{g_1}} x_i= \sum_{\tau_j\in c_{g_1^{-1}}}x_j,\]we can fix 
an arbitrary bijection :
\[\Psi_1^c : \{\tau_i^k\, :\, \tau_i\in c_{g_1}, 1\leq k\leq x_i\} \to \{\tau_{i'}^{k'}\, :\, \tau_{i'}\in c_{g_1^{-1}}, 1\leq k'\leq x_{i'}\}\]

and put an edge $\tau_i^k \to \tau_{i'}^{k'}$ if and only if $\Psi_1^c(\tau_i^k) = \tau_{i'}^{k'}$ for some $c\in\mathcal C$. Because each vertex has indegree and outdegree 1, it is a (not necessarily connected) Eulerian graph and admits a finite set of cycles covering every vertex exactly once.  

Notice that by construction, if $G_1$ has an edge $\tau_i^k \to \tau_{i'}^{k'}$, then $\Gamma_1$ has an edge $\tau_i \to \tau_{i'}$. Therefore each cycle of $G_1$ can be sent on a cycle in $\Gamma_1$ though the projection $\tau_i^k \mapsto \tau_i$. In this way, project the finite set of cycles obtained above and decompose them into simple cycles of $\Gamma_1$. Denote $x_{1,j}$ the total number of each simple cycle $\omega_1^j$ obtained in this way. 

Because each tile $\tau_i$ was present in $G_1$ as a vertex in $x_i$ copies, we have for every $i$: $\sum_{j=1}^{\#\mathcal{SC}(\Gamma_1)} x_{1,j}|\omega_{1}^{j}|_{\tau_i} = x_i$.

Now apply the same argument for each generator $g_2\dots g_n$ and the variables $(x_{i,j})$ thus obtained are a solution to equation $(\star\star)$.
\end{proof}

\section{Necessary conditions for tiling arbitrary groups}

Since the above conditions apply on sets of Wang tiles or set of graphs, they actually are conditions on a family of $G$-SFT where $G$ range over all groups with a fixed number of generators. The following proposition relates the properties of these SFT. It can be found (under a different form) in \cite{CP15} (Proposition~10 and remark below)

\begin{proposition}\label{prop:grupos}
Let $G_{1}=\langle \{g_{1}\dots g_{d}\}|\mathcal R\rangle$, $G_{2}=\langle \{g_{1}\dots g_{d}\}|\mathcal R^\prime\rangle$ be finitely generated groups, with $\mathcal R^\prime \subset \mathcal R$. Consider the canonical surjective morphism $\pi: G_{2}\rightarrow G_{1}$ defined by $\pi(g_{i})=g_{i}$, $\forall 1\leq i\leq d$. Let $\Phi : \A^{G_1} \to \A^{G_2} $ be defined by $\Phi(x)_{g} = x_{\pi(g)}$. Let $X_1$ and $X_2$ be the corresponding $G_1$-NNSFT and $G_2$-NNSFT respectively, such that $X_{2}$ has the same local rules that $X_{1}$.

We have:
\begin{enumerate}
\item If $x$ is a valid tiling for $X_1$ then $\Phi(x)$ is a valid tiling for $X_2$.
\item If $x$ is weakly periodic then $\Phi(x)$ is weakly periodic. In particular, if $X_1$ admits a weakly periodic tiling, then $X_2$ admits a weakly periodic tiling.
\item If $x$ is weakly aperiodic then $\Phi(x)$ is weakly aperiodic. In particular, if $X_1$ admits a weakly aperiodic tiling, then $X_2$ admits a weakly aperiodic tiling.
\end{enumerate}
\end{proposition}

The strong properties are not preserved by $\Phi$, but of course the image of a strongly (a)periodic tiling remains weakly (a)periodic. Stronger versions with different hypotheses can be found in \cite{CP15, EJ15b}.

\begin{proof} 
\begin{enumerate}
 \item Since $X_2$ is an NNSFT, it is enough to check that, for all $h\in G_2$ and all $1\leq i\leq d$, $\Phi(x)_h\to \Phi(x)_{hg_i}$ is an edge in $\Gamma_i$, that is to say, that it is not a forbidden pattern for $X_2$. By definition of $\Phi$, $\Phi(x)_h = x_{\pi(h)}$ and $\Phi(x)_{hg_i} = x_{\pi(h)\pi(g_i)} = x_{\pi(h)g_i}$. Because $x$ is a valid tiling for $X_1$, we have that $x_{\pi(h)} \to x_{\pi(h)g_i}$ is an edge in $\Gamma_i$, which proves the result.
 
 \item If $x$ is a weakly periodic tiling in $X_{1}$, then $\stab_\sigma(x)$ is nontrivial by definition. We have:
 \begin{align*}
 \stab_\sigma(\Phi(x)) &= \{g\in G_2\,:\,\forall h\in G_2, \Phi(x)_{hg} = \Phi(x)_h\}\\
 &  =\{g\in G_2\,:\,\forall h\in G_2, x_{\pi(h)\pi(g)} = x_{\pi(h)}\}
 \end{align*}
 
Since $\pi$ is surjective, this means that $\pi(\stab_\sigma(\Phi(x))) = \stab_\sigma(x)$. $\stab_\sigma(x)$ is nontrivial so $\stab_\sigma(\Phi(x)) = \pi^{-1}(\stab_\sigma(x))$ is nontrivial as well. 

\item If $x$ is a weakly aperiodic tiling in $X_{1}$, then $\stab_\sigma(x)$ does not have finite index. The canonical morphism $\pi : G_2 \to G_1$ yields a morphism on the quotient:
\[\tilde\pi : G_2/\pi^{-1}(\stab_\sigma(x)) \to G_1/\stab_\sigma(x)\]
and $\tilde \pi$ is surjective since $\pi$ is surjective. Remember that $\stab_\sigma(\Phi(x)) = \pi^{-1}(\stab_\sigma(x))$ by the previous point. Since $\stab_\sigma(x)$ does not have finite index, $G_1/\stab_\sigma(x)$ is infinite, so $G_2/\pi^{-1}(\stab_\sigma(x))$ is infinite as well, and $\stab_\sigma(\Phi(x)) = \pi^{-1}(\stab_\sigma(x))$ does not have finite index.
\end{enumerate}
\end{proof}

\begin{remark}
In the last proposition, the converse of the point (1) does not hold. For instance, if consider $G=\mathbb{Z}^{2}=\langle g_{1},g_{2}\mid g_{1}g_{2}g_{1}^{-1}g_{2}^{-1}\rangle$. Example~\ref{ex:piantadosi} provided an example of a set of graphs that satisfied condition $(\star)$ (so the corresponding $\Fg{2}$-NNSFT admits a valid tiling) but does not satisfies the conditions $(\star\star)$ (so the corresponding $\Z{2}$-NNSFT does not admit any valid tiling).

To understand why, notice that $ker(\pi)$ contains $g_{1}g_{2}g_{1}^{-1}g_{2}^{-1}$, so if a tiling $x \in \A^{\Fg2}$ is such that $x_{1_{\Fg2}}\neq x_{g_{1}g_{2}g_{1}^{-1}g_{2}^{-1}}$, then $\Phi^{-1}(x) = \emptyset$. If this happens for all $x\in X_2$ then $X_1$ is empty.
\end{remark}

\begin{corollary}\label{cor:arbitrary}
Let $\Gamma_1,\dots, \Gamma_d$ be a set of graphs that does not satisfy the condition $(\star)$. Then the corresponding $G$-NNSFT is empty for an arbitrary group $G$ with $d$ generators.
\end{corollary}

\begin{proof}
If there was a valid tiling in $G = \langle g_1,\dots, g_d \mid \mathcal R \rangle$ then, applying Proposition~\ref{prop:grupos}, we would obtain a tiling on $\Fg{d} = \langle g_1,\dots, g_d \mid \emptyset \rangle$, which is in contradiction with Theorem~\ref{thm:star}.
\end{proof}

\section{Necessary conditions for tiling amenable groups}

\begin{definition}[Følner sequence]
Let $G$ be a finitely generated group. A \emph{Følner sequence} for $G$ is a sequence of finite subsets $S_n\subset G$ such that:
F
\[G = \bigcup_nS_n\quad\text{and}\quad\forall g\in G, \frac{\#(S_n g\triangle S_n)}{\#S_n}\xrightarrow[n\to \infty]{}0,\]
where $S_n g = \{hg\ :\ h\in S_n\}$ and $A\triangle B = (A\backslash B)\cup (B\backslash A)$ is the symmetric difference. 
\end{definition}

In the previous definition, it is easy to see that the second condition only has to be checked for $g$ in a finite generating set. The set $S_ng\triangle S_n$ can be understood as the border of $S_n$, so an element of a Følner sequence must have a small border relative to its interior.

\begin{definition}[Amenable group]
A finitely generated group $G$ is \emph{amenable} if it admits a Følner sequence.
\end{definition}

A few examples : 
\begin{itemize}
\item $\mathbb Z^d$ is amenable and a Følner sequence is given by $S_n = [-n, n]^d$. Indeed, if $(g_i)_{1\leq i\leq d}$ is the canonical set of generators, then $\#S_n = (2n+1)^d$ and $\#(S_n+g_i)\triangle S_n = 2\cdot(2n+1)^{d-1}$.
\item $\Fg{d}$ for $d\geq 2$ is not amenable. In particular, the balls $S_n$ of radius $n$ - that is, reduced\footnote{with no $g_i^{-1}g_i$ or $g_ig_i^{-1}$ factors} words of length $\leq n$ on the set of generators $(g_i)_{1\leq i\leq d}$ - are not a Følner sequence. Indeed, one can easily check that $\#S_n = \Omega(d^n)$ and $\#(S_n g_i\triangle S_n) = \Omega(d^n)$.
\end{itemize}

The following theorem was conjectured in \cite{EJ15b}, Section~3.1.

\begin{theorem}[Heuristic for tiling an amenable group]\label{thm:amenable}
Let $G$ be a finitely generated amenable group, $S$ a finite set of generators, and $T$ a set of Wang tiles. 

If there is a tiling of $G$ with the tiles $T$, then condition $(\star\star)$ (or equivalently $(\star\star)^\prime$) is satisfied.
\end{theorem}

\begin{proof}
Let $x \in T^G$ be a tiling of $G$ and $S_n$ be a Følner sequence for $G$. Using notations from Definition~\ref{def:starstar}, for a color $c\in\mathcal C$ and a generator $g\in S$, $c_g$ is the set of tiles $\tau$ such that $\tau(g)=c$. 

For any $h\in S_n\cap S_n g^{-1}$, we have $x_h\in c_g \Leftrightarrow x_{hg}\in c_{g^{-1}}$ (and in this case, $hg \in S_n\cap S_n g$). This means that, for all $c\in\mathcal C, g\in S$ and $n\in \mathbb N$:
\[\#\{h\in S_n\cap S_n g^{-1}\ :\ x_h\in c_g\} = \#\{h\in S_n\cap S_n g\ :\ x_h\in c_{g^{-1}}\}\]
so in particular $|\#\{h\in S_n\ :\ x_h\in c_g\} - \#\{h\in S_n\ :\ x_h\in c_{g^{-1}}\}| \leq \#S_ng\triangle S_n+\#S_ng^{-1}\triangle S_n$.

For each tile $\tau_i$, let $x^n_i = \frac{\#\{h\in S_n\ :\ x_h = \tau_i\}}{\#S_n}$. The previous computation implies that:
\[\forall g\in S, \forall c\in\mathcal C, \left|\sum_{\tau_i\in c_g} x^n_i - \sum_{\tau_j\in c_{g^{-1}}}x^n_j\right| \leq \frac{\#S_ng\triangle S_n}{\#S_n}+\frac{\#S_ng^{-1}\triangle S_n}{\#S_n}.\]

Notice that the right-hand side tends to $0$ as $n$ tends to infinity by definition of a Følner sequence. Consider the sequence of vectors $((x_i^n)_i)_{n\in\N}$ and, by compacity, let $(x_i)$ be any limit point of this sequence. Since $\sum_i x^n_i = 1$ for all $n$ by definition, $\sum_i x_i=1$ as well, and we have
\[\forall g\in S, \forall c\in\mathcal C, \sum_{\tau_i\in c_g} x_i = \sum_{\tau_j\in c_{g^{-1}}}x_j,\]
so $(x_i)$ is a nontrivial solution to Equation $(\star\star)$. Condition $(\star\star)^\prime$ follows by Theorem~\ref{thm:equivalence}.\end{proof}

\section{Counterexamples}

It is clear that none of the $(\star)$, $(\star\star)$ or $(\star\star)^\prime$ conditions can be a sufficient condition to admit a $\mathbb Z^d$-tiling, since it would be a decision procedure for the Domino problem; this argument applies to any group where the Domino problem is undecidable. For completeness, we provide explicit counterexamples for any non-free finitely generated group.

\begin{theorem}
Let $G$ be an arbitrary finitely generated group. If $G$ is not free, then there exists a Wang tile set that satisfies the three conditions $(\star)$, $(\star\star)$ and $(\star\star)^\prime$ and such that the corresponding $G$-Wang subshift is empty.
\end{theorem}

\begin{proof}
Write $G = \langle g_1\dots g_d\mid \mathcal R\rangle$, and take $r_1 : w_1\dots w_n \in \mathcal R$, with $w_1\dots w_n$ a reduced word on generators $g_1\dots g_d$ (no generator is next to its inverse).

We build a family of graphs $\Gamma_d$ on vertices $\{0,\dots, n\}$ with the following edges:
\[\forall i\leq n, \left\{\begin{array}{l}\text{if }w_i = g_j\text{, then }\Gamma_j\text{ has an edge }i-1 \to i;\\
\text{if }w_i = g_j^{-1}\text{, then }\Gamma_j\text{ has an edge }i\to i-1.\\
\end{array}\right.\]
Notice that every vertex has indegree and outdegree at most $1$ and we did not create any cycle in the process, so we can complete every $\Gamma_j$ to be isomorphic to a $n$-cycle graph $C_n$.

Now we define a set of $n+1$ Wang tiles on $n+1$ colors $\{0\dots n\}$ as follows. Tile $\tau_i$ has the following colors: for all $j$, $g_j^{-1}\to i$ and $g_j\to k$ if there is an edge $\tau_i\to \tau_k$ in $\Gamma_j$.

\begin{example}
For $\mathbb Z^2$, we have $r_1: g_1g_2g_1^{-1}g_2^{-1} = 1$. Therefore $\Gamma_1$ contains $0\to1$ and $3\to2$, 
and $\Gamma_2$ contains $1\to2$ and $4\to3$. One possible completion for $\Gamma_1$ and $\Gamma_2$ is the following:

\begin{center}
\begin{tikzpicture}[scale=2]
	\foreach \i/\x/\y in {0/0/1,1/.95/.3,4/.58/-.8,3/-.58/-.8,2/-.95/.3} {
		\node (V\i) at (\x,\y) [shape=circle,draw=black] {$\tau_\i$};
	}
	\draw[-latex] (V0) [bend left] to (V1);
	\draw[-latex] (V1) [bend left] to (V4);
	\draw[-latex] (V4) [bend left] to (V3);
	\draw[-latex] (V3) [bend left] to (V2);
	\draw[-latex] (V2) [bend left] to (V0);
	\node at (-1,1) {$\Gamma_1:$};
	
	\begin{scope}[shift={(3,0)}]
	\foreach \i/\x/\y in {0/0/1,1/.95/.3,2/.58/-.8,4/-.58/-.8,3/-.95/.3} {
		\node (V\i) at (\x,\y) [shape=circle,draw=black] {$\tau_\i$};
	}
	\draw[-latex] (V0) [bend left] to (V1);
	\draw[-latex] (V1) [bend left] to (V2);
	\draw[-latex] (V2) [bend left] to (V4);
	\draw[-latex] (V4) [bend left] to (V3);
	\draw[-latex] (V3) [bend left] to (V0);
	\node at (-1,1) {$\Gamma_2:$};					
\end{scope}
\end{tikzpicture}
\end{center}	
and the corresponding set of Wang tiles:
\begin{center}
\begin{tikzpicture}
\begin{scope}[shift={(0,0)}]
\node at (1,2.6) {$\tau_0$};
\draw (0,0) -- (2,0) -- (2,2) -- (0,2) -- (0,0) -- (2,2) (0,2) -- (2,0);
\node at (0.3,1) {$0$};
\node at (1.7,1) {$1$};
\node at (1,1.7) {$1$};
\node at (1,0.3) {$0$};
\end{scope}
\begin{scope}[shift={(2.7,0)}]
\node at (1,2.6) {$\tau_1$};
\draw (0,0) -- (2,0) -- (2,2) -- (0,2) -- (0,0) -- (2,2) (0,2) -- (2,0);
\node at (0.3,1) {$1$};
\node at (1.7,1) {$4$};
\node at (1,1.7) {$2$};
\node at (1,0.3) {$1$};
\end{scope}

\begin{scope}[shift={(5.4,0)}]
\node at (1,2.6) {$\tau_2$};
\draw (0,0) -- (2,0) -- (2,2) -- (0,2) -- (0,0) -- (2,2) (0,2) -- (2,0);
\node at (0.3,1) {$2$};
\node at (1.7,1) {$0$};
\node at (1,1.7) {$4$};
\node at (1,0.3) {$2$};
\end{scope}

\begin{scope}[shift={(8.1,0)}]
\node at (1,2.6) {$\tau_3$};
\draw (0,0) -- (2,0) -- (2,2) -- (0,2) -- (0,0) -- (2,2) (0,2) -- (2,0);
\node at (0.3,1) {$3$};
\node at (1.7,1) {$2$};
\node at (1,1.7) {$0$};
\node at (1,0.3) {$3$};
\end{scope}

\begin{scope}[shift={(10.8,0)}]
\node at (1,2.6) {$\tau_4$};
\draw (0,0) -- (2,0) -- (2,2) -- (0,2) -- (0,0) -- (2,2) (0,2) -- (2,0);
\node at (0.3,1) {$4$};
\node at (1.7,1) {$3$};
\node at (1,1.7) {$3$};
\node at (1,0.3) {$4$};
\end{scope}

\end{tikzpicture}
\end{center}
\end{example}

This tiling satisfies condition $(\star\star)^\prime$ since we can assign the same weight $\frac 1n$ to each tile.

It is clear that a tiling $x$ of $G$ using tiles $\tau_0,\dots \tau_n$ must contain every tile. Assume w.l.o.g that $x_1 = \tau_0$. By construction we must have $x_{w_1} = \tau_1$, $x_{w_1w_2} = \tau_2$, and by an easy induction $x_w = \tau_n$. But since $w=1$ in $G$, we have $\tau_0 = x_1 = x_w = \tau_n$, a contradiction. Therefore there is no tiling of $G$ using tiles $\tau_0,\dots \tau_n$.
\end{proof}

\section{Conclusion}

We would like to mention the two following conjectures that relate the fact of admitting a valid (periodic) tiling and the underlying group structure:
\begin{conjecture}[\cite{BS13}]\label{conj:1}
A finitely generated group has a decidable domino problem if and only if it is virtually free.
\end{conjecture}

\begin{conjecture}[\cite{CP15}]
A finitely generated group has an SFT with no strongly periodic point if and only if it is not virtually cyclic.
\end{conjecture}

In both cases, the ``if'' direction is proven and the ``only if'' direction is open.  

\section*{Acknowledgements}
The first author would like to thank Pascal Vanier and Emmanuel Jeandel for providing access to a preprint of \cite{JV19} and help in understanding \cite{CG14}. We are grateful to an anonymous referee for many helpful remarks on a previous version.

\bibliographystyle{plain}
\bibliography{biblio}

\end{document}